\documentclass[12pt,a4paper,oneside]{amsart}

\usepackage{amsfonts, amsmath, amssymb, amsthm, amscd}
\usepackage{hyperref}
\hypersetup{
  colorlinks   = true, 
  urlcolor     = blue, 
  linkcolor    = blue, 
  citecolor   = red 
}
\usepackage{anysize}
\marginsize{2cm}{2cm}{1.5cm}{1.5cm}
\usepackage[pdftex]{graphicx}

\newtheorem{theorem}{Theorem}[section]
\newtheorem{lemma}[theorem]{Lemma}
\newtheorem{claim}[theorem]{Claim}

\theoremstyle{definition}
\newtheorem{definition}[theorem]{Definition}

\theoremstyle{remark}
\newtheorem{remark}[theorem]{Remark}

\newtheorem{example}[theorem]{Example}

\numberwithin{equation}{section}

\renewcommand{\epsilon}{\varepsilon}

\renewcommand{\phi}{\varphi}

\newcounter{fig}

\title{Equipartition of a segment}

\author{Sergey Avvakumov{$^\spadesuit$}}
\author{Roman~Karasev{$^\clubsuit$}}

\thanks{{$^\spadesuit$} Supported by the European Research Council under the European Union's Seventh Framework Programme ERC Grant agreement ERC StG 716424 -- CASe}

\thanks{{$^\clubsuit$} Supported by the Russian Foundation for Basic Research grant 19-01-00169}

\address{Sergey~Avvakumov, Department of Mathematical Sciences, University of Copenhagen, Universitetspark 5, 2100 Copenhagen, Denmark}
\email{savvakumov@gmail.com}

\address{Roman~Karasev, Institute for Information Transmission Problems RAS, Bolshoy Karetny per. 19, Moscow, Russia 127994\newline \indent
Moscow Institute of Physics and Technology, Institutskiy per. 9, Dolgoprudny, Russia 141700}
\email{r\_n\_karasev@mail.ru}
\urladdr{http://www.rkarasev.ru/en/}


\begin{document}

\begin{abstract}
We prove that, for any positive integer $m$, a segment may be partitioned into $m$ possibly degenerate or empty segments with equal values of a continuous function $f$ evaluated on segments, assuming that $f$ may take positive and negative values, but its value on degenerate or empty segments is zero. 
\end{abstract}

\maketitle

\section{Introduction}

The mathematical theory of fair division develops along two main lines of research. On the one hand, it looks for partitions of a body or measure into $m$ pieces (the positive integer $m$ is fixed) of a certain shape and equal in some sense, e.g. convex polygons of identical area. An early example is the ham sandwich theorem \cite{steinhaus1945,stonetukey1942} about equipartitioning several measures by hyperplanes. More recent examples are the Nandakumar and Ramana Rao conjecture \cite{nandakumar2008} that every convex planar polygon can be partitioned into $m$ convex polygons of equal area and perimeter, solved in \cite{aak2018}, and higher-dimensional analogues of the Nandakumar and Ramana Rao problem that were solved in \cite{bz2014,ahk2014} under the assumption that $m$ is a prime power. See also \cite{soberon2012} for a result in between the ham sandwich theorem and the Nandakumar and Ramana Rao problem.

On the other hand, the theory of fair division contributes key existence results to the concept of fairness favored by economists and many social scientists, known as Envy Freeness, that is, each one of $m$ agents compares pieces of the partition in her own way. We look for an $m$-partition where each agent gets, in her own view, one of the best pieces. In this paper all the agents use the same utility function. For further references on it, see the foundational work of Gale \cite{gale1984} and popular reviews \cite{bbs2009,su1999}.

Here we prove a topological property for partitions of a segment that we believe to be useful for both types of results just mentioned. There is a single agent who evaluates each subsegment $[a,b]$ of $[0,1]$ by a continuous utility function $f(a,b)$ such that $f(a,a)=0$ for all $a$. Apart from the continuity requirement (understood as continuity of a function in two real variables), $f$ is very general, in particular it can take both positive and negative values. We show the existence of an $m$-partition of $[0,1]$ into subsegments all of equal utility. This is a key ingredient used in the companion paper \cite{bogomolnaia2020},
where it implies the existence of a universal Fair Guarantee, which is a utility level that can be achieved simultaneously by any $m$ agents, each with her own utility function.

\begin{theorem}
\label{theorem:main}
Let $\mathcal I$ be the space of possibly degenerate subsegments 
\[
[a,b]\subseteq [0,1],\quad 0\le a\le b\le 1. 
\]
Assume we have a continuous function $f : \mathcal I\to \mathbb R$ such that for degenerate segments we have $f([a,a])\equiv 0$ for all $a\in [0,1]$. Then for any positive integer $m$ it is possible to partition the segment $[0,1]$ into $m$ possibly degenerate segments
\[
[0,1] = [0,x_1]\cup [x_1,x_2]\cup\dots \cup[x_{m-1},1],\quad 0\le x_1\le x_2\le \dots \le x_{m-1}\le 1,
\]
so that 
\[
f([0,x_1])=f([x_1,x_2])=\dots = f([x_{m-1},1]).
\]
\end{theorem}

Let us first comment on the novelty of this result. This theorem might look like a particular case $d=1$ of \cite[Theorem 6.1]{aak2018}, but it is not. The difference is that in Theorem \ref{theorem:main} we additionally require a certain behavior of $f$ on degenerate segments, while in \cite[Theorem 6.1]{aak2018} the assumption $d\ge 2$ eliminates the necessity to consider degenerate parts. The function of a convex body $f$ in the proof of \cite[Theorem 6.1]{aak2018} is only applied to convex bodies of a certain positive volume (measure), thus excluding the need to consider degenerate parts and extend the function to them.

On the level of proofs, we use the same approach of using ``nice multivalued functions'', adapting  \cite[Lemma 4.2]{aak2018} to our problem in the form of Lemma \ref{lemma:function-to-function} below. The proof of this lemma is also very similar to the proof of \cite[Lemma 4.2]{aak2018}, but in this paper the proof is made self-contained and independent of the very technical results of \cite{bz2014}, which are used in \cite{aak2018}, with the help of a simpler configuration space.

Let us also comment on the previously known particular cases of Theorem \ref{theorem:main}. The case of non-negative $f$ in Theorem \ref{theorem:main} follows from the Knaster--Kuratowski--Mazurkiewicz theorem \cite{kkm1929} in a standard way. The case of $m$ a prime power and $f$ of varying sign follows from the more general result of \cite{avvakumov2019}, the case of prime $m$ following from \cite{meunier-zerbib2018}. The case of $f$ additive on segments is an elementary exercise. Hence \emph{the new case here is when the sign of $f$ varies, $f$ is not additive, and $m$ is not a prime power.} Of course, our proof is also an essentially new proof for those known particular cases.

A generalization of Theorem \ref{theorem:main} for envy-free divisions, when $m$ players divide a segment into $m$ possibly empty parts and each of the players wants to receive one of the best parts according to his/her individual function $f_i : \mathcal I\to\mathbb R$, is open for $m$ not a prime power. See further explanations and definitions on envy-free division of the segment in \cite{segal2018,meunier-zerbib2018,avvakumov2019}.

\subsection*{Acknowledgments} 
We thank Pablo Sober\'on, Peter Landweber, and Herv\'e Moulin for numerous remarks that helped improve the exposition. We also thank the anonymous referees for convincing us to rewrite the paper in a more elementary and self-contained way.

\section{Reduction to the lemma on multivalued functions}

First, we pass from single-valued functions to multi-valued functions. After rescaling we may assume that $f$ in the statement of the theorem takes values in $(-1,1)$. Let the \emph{cylinder} be the set $\mathcal I\times [-1,1]$. 

\begin{definition}
A \emph{nice} multi-valued function $\mathcal I\to [-1,1]$ is a compact subset $Z\subset \mathcal I\times (-1,1)$ (also called the \emph{graph of the multi-valued function}) that \emph{separates the top from the bottom}, that is, the sets $\mathcal I\times \{-1\}$ and $\mathcal I\times \{1\}$ belong to different connected components of $\mathcal I\times [-1,1]\setminus Z$.
\end{definition}

The following lemma allows us to represent nice multi-valued functions by single-valued continuous functions on the whole cylinder. This will be needed to build maps out of single-valued functions and apply Borsuk--Ulam-type arguments to the maps, see the proof of Claim \ref{claim-nonzero-cycle} below.

\begin{lemma}
\label{lemma:two-definitions}
$(a)$ For any nice multi-valued function of $\mathcal I$, its graph is the zero set of an ordinary continuous function $\phi : \mathcal I\times [-1,1]\to\mathbb R$ such that $\phi(\mathcal I\times \{-1\})<0$ and $\phi(\mathcal I\times \{1\})>0$.

$(b)$ For any ordinary continuous function $\phi : \mathcal I\times [-1,1]\to\mathbb R$ such that $\phi(\mathcal I\times \{-1\})<0$ and $\phi(\mathcal I\times \{1\})>0$, its zero set is a graph of a nice multi-valued function.
\end{lemma}

\begin{proof}
Claim (b) is trivial, so we prove (a). For a graph $Z\subset \mathcal I\times (-1,1)$ of a nice multi-valued function let $\phi$ be the distance to $Z$ with a sign. It is possible to choose the sign arbitrarily for each connected component of $\mathcal I\times [-1,1]\setminus Z$; any such signed distance function is continuous. The requirement for the sign of $\phi$ is achieved if we choose the sign of $\phi$ positive on the top, negative on the bottom, and arbitrarily on other connected components. Note that the ``nice'' property allows choosing the sign on the top and on the bottom independently.
\end{proof}

In what follows we pass back and forth between the two points of view on multi-valued functions using Lemma \ref{lemma:two-definitions}.
The function $f$ from the statement of Theorem \ref{theorem:main} can be considered as a nice multi-valued function with $\phi([a,b], y) = y - f([a,b])$. The boundary assumption $f([a,a])\equiv 0$ for all $a$ means that $\phi([a,a], y) \equiv y$ for all $a$. 

\begin{lemma}[A modification of Lemma 4.2 from \cite{aak2018}]
\label{lemma:function-to-function}
Assume a continuous $\phi : \mathcal I\times [-1,1]\to\mathbb R$ corresponds to a nice multi-valued function and $\phi([a,a], y)\equiv y$ for all $y$. Let $p$ be a prime. Then there exists another nice multi-valued function, represented by a continuous $\psi : \mathcal I\times [-1,1]\to \mathbb R$ such that $\psi([a,a], y)\equiv y$ for all $y$ and, whenever $I\in\mathcal I$ and $y\in (-1,1)$ satisfy
\[
\psi(I, y) = 0
\]
then there exists a partition $I = I_1\cup \dots \cup I_p$ into possibly degenerate segments such that
\begin{equation}
\label{equation:equalized-p}
\phi(I_1, y) = \dots = \phi(I_p, y) = 0.
\end{equation}
\end{lemma}

\begin{proof}[Proof of Theorem \ref{theorem:main} assuming Lemma \ref{lemma:function-to-function}]
Let $m = p_1p_2\dots p_n$ be a decomposition into primes. Let $\phi_1$ be the initial single-valued function $f$. Apply the lemma to $\phi_1$ and $p_1$ to obtain a nice multi-valued function $\phi_2$. Then apply the lemma to $\phi_2$ and $p_2$ and so on. The final function $\phi_{n+1}$ will be a nice mutli-valued function of a segment.

From the definition it follows that a nice multi-valued function assigns at least one value to any segment. Hence there exists $y\in (-1,1)$ such that 
\[
\phi_{n+1}([0,1], y) = 0.
\]
In means that $[0,1]$ may be partitioned into $p_n$ possibly degenerate segments $I_1,\ldots, I_{p_n}$ of the same value $y$ of the multi-valued function $\phi_n$,
\[
\phi_n(I_1, y) = \dots = \phi_n(I_{p_n}, y) = 0.
\]
Each of these segments may in turn be partitioned into $p_{n-1}$ segments of the same value $y$ of the multi-valued function $\phi_{n-1}$, and so on. Eventually, we obtain a partition of $[0,1]$ into $m=p_1\cdots p_n$ parts of the same value $y$ of the multi-valued function $\phi_1$, which is in fact the single-valued function $f$ we have started from.
\end{proof}

\section{Parametrizing partitions of a segment into a prime number of parts}
\label{section:parametrization}

Our proof of Lemma \ref{lemma:function-to-function} will use equivariant maps and certain claims of Borsuk--Ulam type. We generally follow \cite{aak2018}, but instead of using the configuration space of \cite{bz2014}, we make a simplification. In this case the partitions of a segment into segments are easier to parametrize more directly, using the idea of \cite{panina2021} with a slightly different construction.

In order to present a proof of Lemma \ref{lemma:function-to-function}, we need to parametrize all possibly degenerate partitions of the segment $[a,b]$ into $p$ parts ($p$ is a prime number). A direct parametrization by the relative lengths of the segments, 
\[
t_1,\ldots, t_p\ge 0,\quad t_1 + \dots + t_p = 1,
\]
produces the simplex $\Delta^{p-1}$. If the segment $[a,b]$ is degenerate, $a=b$, then we think of all such partitions as the partition of the degenerate segment into degenerate segments. But we still distinguish which of the degenerate segments has its corresponding $t_i=0$ and which has $t_i>0$.

Let us additionally label the parts of the partition by numbers from $1$ to $p$, where part $i$ is labeled by $\sigma(i)$. Thus we obtain the space $\Delta^{p-1}\times\mathfrak S_p$, where $\mathfrak S_p$ is the group of permutations of $\{1,\ldots,p\}$. So far our configuration space is just a disjoint union of $p!$ simplices. 

After that we identify the pairs $(t,\sigma)\sim (t',\sigma')$ if the labelings $\sigma$ and $\sigma'$ become the same sequence of integers after erasing the labels corresponding to degenerate segments with $t_i=0$ and $t'_j=0$. This results in identifying some faces of $\Delta^{p-1}\times\{\sigma\}$ and $\Delta^{p-1}\times \{\sigma'\}$. 

The obtained space $Q_p$ is a $(p-1)$-dimensional cell complex. More precisely, this is a $\Delta$-complex in the sense of \cite[Section~2.1]{hatcher2001}. This complex becomes a \emph{simplicial complex}, that is, a union of faces of a simplex spanned by its vertices, after taking the barycentric subdivision of every simplex in its construction. 

One may also think of $Q_p$ as partitions of $[a,b]$ into at most $p$ non-degenerate parts with the parts carrying distinct labels from $\{1,\ldots, p\}$. This space $Q_p$ has the natural action of the cyclic group $G_p\subseteq \mathfrak S_p$ by cyclic permutations of the labels. Since a partition uses at least one label, it is clear that this action of $G_p$ is \emph{free}, that is, for $g\in G_p$, $g\neq e$, and any $\xi\in Q_p$, we have $g\xi\neq \xi$.

An \emph{orientation} of a simplex is an order of its vertices up to even permutations. If the order of vertices of a simplex $F$ is $(v_0,\ldots,v_k)$ then its boundary is oriented so that a boundary face $(v_0,\ldots,\widehat{v_i},\ldots,v_k)$ (omitted $v_i$) has the given orientation for even $i$ and the opposite orientation for odd $i$.

\begin{lemma}
\label{lemma:pseudomanifold}
Consider the orientation of all $(p-1)$-dimensional cells of $Q_p$ by the order of the labels corresponding to the cell. This orientation makes $Q_p$ a \emph{pseudomanifold} modulo $p$, that is, for every $(p-2)$-dimensional cell $F'$ with any orientation one has the sum over cells containing $F'$
\[
\sum_{F\supset F'} [F : F'] \equiv 0 \mod p,
\] 
where $[F : F']$ is the sign by which the orientation of $F'$ differs from the orientation of $\partial F$.
\end{lemma}
\begin{proof}
The generic points of the cell $F'$ correspond to partitions of $[a,b]$ into $p-1$ non-degenerate parts and a certain fixed order of assigned labels from $\{1,\ldots,p\}$, one label $k\in\{1,\ldots,p\}$ is not assigned. Each cell $F\supset F'$ has the property that generic points of $F$ are partitions of $[a,b]$ into $p$ non-degenerate parts and a certain fixed order of assigned labels from $\{1,\ldots,p\}$, all labels assigned. 

Since we orient $F$ and $F'$ by the order of the labels, the sign $[F : F']$ equals $(-1)^{k-1}$ by the standard convention on the orientation of the boundary of a simplex. Hence this sign does not depend on $F$ and the sum in the statement of the lemma equals $(-1)^{k-1}p$, since there are precisely $p$ such $F$ containing $F'$, corresponding to inserting the label $k$ into any of the $p$ places of the sequence labeling $F'$.
\end{proof}

\begin{lemma}
\label{lemma:orientation-action}
The orientation of all $(p-1)$-dimensional cells of $Q_p$ by the order of the labels corresponding to the cell is invariant with respect to the action of $G_p$ when $p$ is an odd prime, and changes under the action of the nontrivial element of $G_2$ when $p=2$. The collection of $(p-1)$-faces of $Q_p$ with given orientations is then a $G_p$-equivariant $(p-1)$-dimensional cellular cycle modulo $p$.
\end{lemma}

\begin{proof}
The proof follows from the fact that $G_p$, viewed as a subgroup of the permutations, consists of even permutations when $p$ is odd. 

For $p=2$ there is no need in choosing orientations since $+1\equiv -1 \mod 2$, every vertex of $Q_2$ has two edges connected to it, and equivariance modulo $2$ is trivial.
\end{proof}

\begin{example}
$Q_2$ is constructed from two segments with endpoints labeled by $\{1,2\}$, glued according to their labels. Topologically $Q_2$ is a circle, though our choice of orientations on the segments according to their labels does not produce an orientation of the circle, since the orientations do not match where the segments are glued. We may only say that with such a ``wrong'' orientation this circle becomes a pseudomanifold modulo $2$, see Figure \ref{figure:Q2}.
\end{example}

\begin{figure}[ht]
\center
\includegraphics[width=150mm]{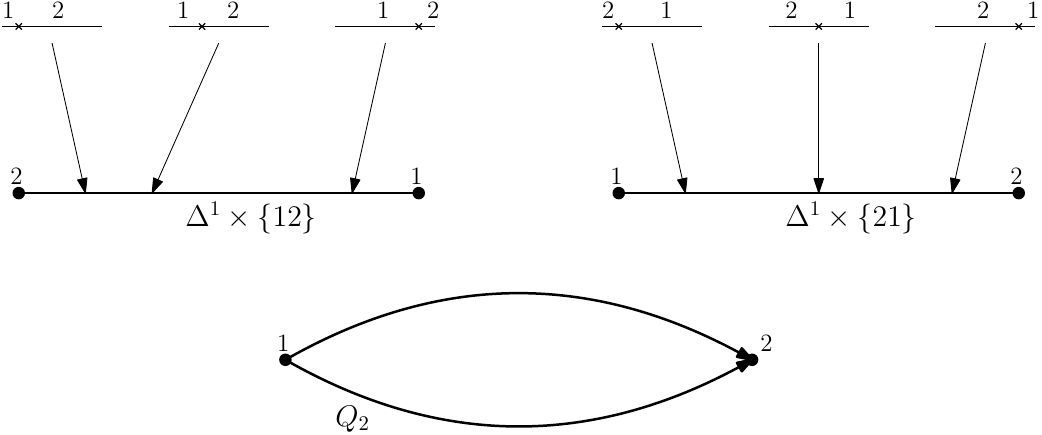}
\caption{Above: Labeled partitions of the segment into two segments and corresponding points in the $1$-dimensional simplices of $Q_2$.
\newline 
Below: The complex $Q_2$ with orientation of its simplices.}
\label{figure:Q2}
\end{figure}

\begin{example}
$Q_3$ is constructed from six triangles, each having labels $\{1,2,3\}$ on its vertices and each having an order of vertices corresponding to the order on the segment that we partition. We may index those triangles as
\[
(1,2,3), (1,3,2), (2,1,3), (2,3,1), (3,1,2), (3,2,1).
\]
The gluing rules mean that $Q_3$ has three vertices with the labels $1,2,3$ and six edges labeled 
\[
(1,2), (2,1), (2,3), (3,2), (1,3), (3,1),
\]
attached to the vertices accordingly. In particular, the $1$-skeleton of $Q_3$ is not a simple graph. The six triangles are glued to the edges  so that, for example, triangle with labels $(3,2,1)$ is glued to the edges $(3,2),(2,1),(3,1)$ according to the labels, see Figure \ref{figure:Q3}.
\end{example}

\begin{figure}[ht]
\center
\includegraphics[width=150mm]{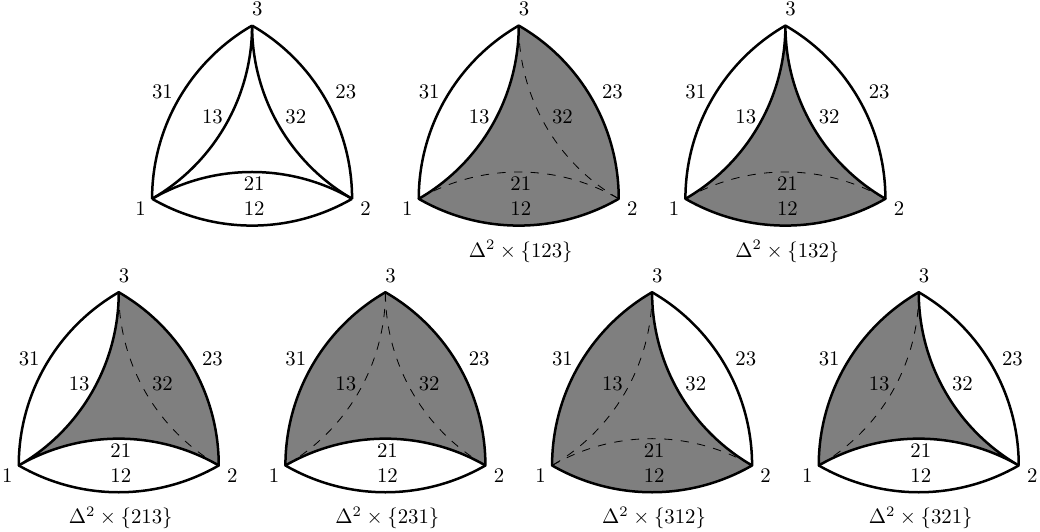}
\caption{The $1$-dimensional skeleton of $Q_3$ and its six triangular faces shown separately.}
\label{figure:Q3}
\end{figure}

We continue to investigate the properties of $Q_p$ that are needed in the proof of Lemma~\ref{lemma:function-to-function}. Let $W_p\subset\mathbb R^p$ be the linear subspace of vectors with the sum of coordinates equal to zero, which is $(p-1)$-dimensional. Let the cyclic group $G_p$ act on $W_p$ by cyclic permutations of the coordinates. This action flips the orientation of $W_p$ only when $p=2$, when we do not care about orientations.

The following construction together with the pseudomanifold property from Lemma~\ref{lemma:pseudomanifold} ensures a Borsuk--Ulam-type property for maps from $Q_p$ to $W_p$. Define the map $\Psi_0 : Q_p\to W_p$ as taking a labeled partition $(t,\sigma)$ to the sequence 
\begin{equation}
\label{equation:test-map}
\left( t_{\sigma^{-1}(1)} - \frac1p, \ldots, t_{\sigma^{-1}(p)} - \frac1p\right).
\end{equation}
This map may be viewed as sending a configuration $\xi\in Q_p$ to the lengths of the segments of the corresponding partition, listed in the order of their labels, and then normalized by subtracting the vector $\left(1/p, \ldots, 1/p\right)$. It clearly agrees with the identifications made in the construction of $Q_p$ and is \emph{equivariant}, that is, commuting with the action of $G_p$ on $Q_p$ and $W_p$. 

We need the standard notion of transversality:
\begin{definition}
For a simplicial complex $Q$, a continuous map $f : Q\to \mathbb R^n$ is \emph{transverse to zero} if for every face $F\subseteq Q$ and every point $x\in F$ such that $f(x)=0$ the map $f$ is linear in a neighborhood of $x$ in $F$ and this linear map is surjective. 
\end{definition}

In particular, the images (under a transverse to zero $f : Q\to \mathbb R^n$) of faces $F\subset Q$ of dimension strictly less than $n$ do not contain $0$. When a continuous map does not have $0$ in its image then it is transverse to zero by definition. In what follows we simply write \emph{transverse} instead of \emph{transverse to zero}; this should not lead to a confusion.

\begin{claim}
\label{claim:approximation}
Suppose that a finite group $G$ acts on a simplicial complex $Q$ freely, and acts on $\mathbb R^n$ linearly. Then any equivariant continuous map $f : Q\to \mathbb R^n$ can be approximated by a transverse equivariant PL map $g : Q\to \mathbb R^n$.
\end{claim}
\begin{proof}
Consider an iterated barycentric subdivision of $Q$. Perturb the $f$ map on the vertices of this iterated barycentric subdivision generically and equivariantly; the latter can be achieved because the action is free. Then extend the map linearly to the faces of the subdivision, obtaining $g : Q\to\mathbb R^n$. 

Since the action of $G$ on $Q$ is free, the vertices of any face $F$ of the barycentric subdivision belong to different $G$-orbits, and hence generically the restriction $g|_F$ is transverse, and the whole $g$ is transverse.

If the number of iterations of barycentric subdivision is sufficiently large, $g$ will approximate $f$ arbitrarily close.
\end{proof}

See also \cite[Lemma~3.4]{avvakumov2019} for a generalization of Claim~\ref{claim:approximation} that does not assume that the action is free.

When $\dim Q = n$ and $f : Q\to \mathbb R^n$ is transverse, at any point $\xi\in f^{-1}(0)$ the derivative $Df$ is a well-defined linear map from the $n$-dimensional tangent space of the $n$-face of $Q$ containing $\xi$ to $\mathbb R^n$. Having an orientation on $n$-faces of $Q$ (as we have for $Q_p$) we define the \emph{local degree} of $f$ at $\xi$ as the sign of the determinant of $Df$ (compare to the more general definition in \cite[Section~2.2]{hatcher2001}).

\begin{lemma}
\label{lemma:test-map}
The test map $\Psi_0 : Q_p\to W_p$ is transverse. The local degrees of $\Psi_0$ at every point of $\Psi_0^{-1}(0)$ are the same. There are in total $p!$ points and $(p-1)!$ of $G_p$-orbits in the set $\Psi_0^{-1}(0)$.
\end{lemma}

\begin{proof}
Evidently, $\Psi_0^{-1}(0)$ corresponds to the partition of $[a,b]$ into parts of equal length labeled in each of $p!$ possible ways. Since the orientation on every $(p-1)$-dimensional face $F$ is given by the order of labels and the map \eqref{equation:test-map} is linear on $F$, one easily sees the map $\Psi_0$ around any point $\xi\in \Psi_0^{-1}(0)$ does not flip the given orientation of $Q_p$ compared to the orientation of $W_p$ and all the signs are $+1$.
\end{proof}

\section{Proof of Lemma \ref{lemma:function-to-function}}

We are going to parametrize partitions of a segment $[a,b]$ into $p$ possibly degenerate segments by the above described $Q_p$. The proof below, except for the ending, almost literally follows the proof of Lemma~4.2 in \cite{aak2018}, but we choose to present the self-contained argument here.

For any $\xi\in Q_p$ and a label $i\in\{1,\ldots, p\}$ we denote by $I_i$ the subsegment of the segment$[a,b]$ having label $i$ in $\xi$. If the label goes to a degenerate segment then this segment is not well-defined, though it is correct to insert such $I_i(\xi)$ into the function $\phi$ from the statement of the lemma, since $\phi$ is the same for all degenerate segments by the assumptions of the lemma. In this notation, the equations on the varying segment $[a,b]$, $\xi\in Q_p$, and $t\in [-1,1]$
\begin{equation}
\label{equation:equalized-p-d}
\phi(I_1(\xi), y) = \dots = \phi(I_p(\xi), y) = 0,
\end{equation}
define a closed subset $S\subset \mathcal I\times Q_p\times [-1,1]$.

The set $S$ is $G_p$-invariant, where $G_p$ acts on $Q_p$ as described and acts trivially on $\mathcal I$ and $[-1,1]$. In other words, the set $S$ is the preimage of zero under the $G_p$-equivariant continuous map
\[
\Phi : \mathcal I\times Q_p \times [-1,1]\to \mathbb R^p,\quad \Phi(I, \xi, y) = \left( \phi(I_1(\xi), y), \phi(I_2(\xi), y), \ldots, \phi(I_p(\xi), y) \right).
\] 

As the first step in understanding $S$, we fix a segment $I\in \mathcal I$ and study the structure of the fiber set
\[
S_I = S\cap \left( \{I\}\times Q_p \times [-1,1] \right). 
\]
Put $\Phi_I = \Phi|_{\{I\}\times Q_p \times [-1,1]}.$

\begin{claim}
\label{claim-nonzero-cycle}
For a transverse $\Phi_I$, the set $S_I$ is a finite point sets representing a nontrivial $G_p$-equivariant $0$-cycle modulo $p$, that is, the local degrees of $\Phi_I$ at the orbits of $S_I$ sum up to a number not divisible by $p$. The projection of the quotient $S_I/G_p$ to the segment $[-1,1]$ is a nontrivial $0$-cycle modulo $p$.
\end{claim}

\begin{proof}
When $\Phi_I$ is transverse, the solution set $S_I$ is a finite number of points by dimensional considerations. If we deform $\Phi_I$ by a $G_p$-equivariant homotopy keeping the boundary conditions on its components $\phi(I_i, y)$ then the solution set $S_I$ changes, but it changes in a definite way. If the homotopy $H : Q_p\times [-1,1]\times [0,1]\to \mathbb R^p$ is transverse (this can be achieved by a small perturbation from Claim \ref{claim:approximation}) then $H^{-1}(0)$ represents a $G_p$-equivariant $1$-dimensional cycle modulo $p$ relative to $Q_p\times [-1,1]\times \{0,1\}\subset Q_p\times [-1,1]\times [0,1]$.

The domain of the homotopy $Q_p\times [-1,1]\times [0,1]$ is a product of a pseudomanifold modulo $p$ (by Lemma~\ref{lemma:pseudomanifold}) and two segments, and therefore it satisfies the pseudomanifold property for faces not contained in its ``boundary modulo $p$'' $Q_p \times \partial([s_0,s_1] \times [-1,1])$. The mentioned ``$1$-dimensional cycle modulo $p$'' property means that $H^{-1}(0)$ is a graph (with possible loops), whose edges have orientations (induced by the orientation of $Q_p\times [-1,1]\times [0,1]$ and $\mathbb R^p$) and whose every vertex, except for those in $Q_p\times [-1,1]\times \{0,1\}$, has algebraically $0\mod p$ edges connected to it. Note that no vertex of $H^{-1}(0)$ is contained in $Q_p\times \{-1,1\}\times [0,1]$ by construction.

The ``$1$-dimensional cycle modulo $p$'' property holds because under the transversality assumption every vertex $v\in H^{-1}(0)$ is a point of intersection of $H^{-1}(0)$ with a codimension $1$ face $F'\subset Q_p\times [-1,1]\times [0,1]$. If this vertex is not in $Q_p\times [-1,1]\times \{0,1\}$ then there are $0\mod p$ full-dimensional faces $F$ containing such $F'$ of codimension $1$, and each $F'$ contributes the same sign to the orientation of $F'$. Under the transversality assumption each $F\supset F'$ contributes an edge of $H^{-1}(0)$ attached to $v$, all such edges contribute the same sign and their total number is $0\mod p$. The vertices of the graph $H^{-1}(0)$ contained in $Q_p\times [-1,1]\times \{0,1\}$ correspond to the start and the end of the homotopy.

$G_p$-equivariance of $H^{-1}(0)$ as a $1$-dimensional cycle modulo $p$ follows from Lemma \ref{lemma:orientation-action} and the equivariance of the map $H$. Again, for $p=2$ the orientation is actually not needed.

Summarizing, the zero set of a transverse $\Phi_I$ changes equivariantly homologously modulo $p$ to itself under $G_p$-equivariant homotopies of the map $\Phi_I$.

Let us present an instance of a transverse map $\Phi_0 : Q_p\times [-1,1] \to \mathbb R^p$ (a test map), which is $G_p$-equivariant, satisfies the boundary conditions that we impose on $\Phi_I$, and for which the set $\Phi_0^{-1}(0)$ is homologically nontrivial. By the above homotopy consideration (connecting $\Phi_0$ to $\Phi_I$ by a convex interpolation of their coordinates), the existence of such a test map implies the homological nontriviality of $S_I$ for any transverse map $\Phi_I$. In order to produce the needed test map, we take the $G_p$-equivariant test map 
\[
\Psi_0 : Q_p \to W_p\subset \mathbb R^p
\]
from Lemma \ref{lemma:test-map}. The transverse preimage of zero $\Psi_0^{-1}(0)$ consists of $(p-1)!\neq 0 \mod p$ $G_p$-orbits, and all orientations (signs) of those points are equal. This verifies the homological nontriviality of $\Psi_0^{-1}(0)$ as a $0$-dimensional $G_p$-equivariant cycle (the number of signed orbits is not divisible by $p$).

We augment $\Psi_0$ to the map (assuming the coordinates of $\Psi_0$ are in the interval $(-1,1)$)
\[
\Phi_0(\xi,y) = \Psi_0(\xi) + \left(y,\ldots,y\right)\in \mathbb R^p.
\]
Then $\Phi_0^{-1}(0) = \Psi_0^{-1}(0)\times \{0\}$ and this preimage is still a nontrivial $G_p$-equivariant $0$-cycle modulo $p$. This finishes the proof of the first part of the claim.

The second part of the claim means that we consider the set of orbits $S_I/G_p$ with coefficients equal to the local degrees of $\Phi_I$. Every orbit projects to a single point $y\in [-1,1]$, and we assign to such $y$ the coefficient equal to the sum of all coefficients of the orbits of $S_I/G_p$ mapped to $y$. Thus we evidently obtain a non-trivial modulo $p$ cycle in the segment $[-1,1]$.
\end{proof}

\begin{remark}
Note that our homotopy observation is an almost direct generalization of Imre B\'ar\'any's geometric proof of the classical Borsuk--Ulam theorem to the case of a pseudomanifold modulo $p$ as a domain. See, for example, \cite{musin2012,klartag2016} where this technique is explained for honest manifolds as domains.
\end{remark}

Now we understand that the set $S_I$ is always non-empty, since were it empty, the map $\Phi_I$ would be transverse by definition and $S_I$ would have to be non-empty by the claim. 

As the second step in our understanding of $S$, we change the segment $I$ in a continuous one-parameteric family $\{I(s)\ |\ s\in [s_0, s_1]\}$ and obtain a $G_p$-equivariant map with one more parameter
\[
\widetilde\Phi : Q_p \times [s_0,s_1] \times [-1,1] \to \mathbb R^p,\quad \widetilde\Phi (\xi, s, y) = \left( \phi(I_1(\xi, s), y), \phi(I_2(\xi, s), y), \ldots, \phi(I_p(\xi, s), y) \right),
\]
where $I_i(\xi,s)$ is the $i$th part of the partition of $I(s)$ corresponding to $Q_p$.

\begin{claim}
\label{claim:one-parameter}
For a family of segments $I(s)$, the set 
\[
S_{I(s)} = S\cap \left( \{ I(s) \}_{s\in [s_0,s_1]} \times Q_p\times [-1,1] \right)
\] 
separates the top $[s_0,s_1]\times \{1\}$ from the bottom $[s_0,s_1]\times \{-1\}$ when projected to the rectangle $[s_0,s_1]\times [-1,1]$.
\end{claim}

\begin{proof}
Assume first that $\widetilde\Phi$ is transverse. The solution set $\widetilde\Phi^{-1}(0)$ then represents a $G_p$-equivariant $1$-dimensional cycle modulo $p$ relative to $Q_p \times \{s_0,s_1\} \times [-1,1]$. As in the beginning of the proof of Claim \ref{claim-nonzero-cycle}, under the transversality assumption $\widetilde\Phi^{-1}(0)$ is a graph formed by oriented paths or loops in the top-dimensional faces of the domain $Q_p \times [s_0,s_1] \times [-1,1]$ whose ends are isolated points of intersection of $\widetilde\Phi^{-1}(0)$ with the $1$-codimensional skeleton of the domain. 

The domain $Q_p \times [s_0,s_1] \times [-1,1]$ is a product of a pseudomanifold modulo $p$ (by Lemma~\ref{lemma:pseudomanifold}) and two segments, and therefore it satisfies the pseudomanifold property for faces not contained in its ``boundary modulo $p$'' $Q_p \times \partial([s_0,s_1] \times [-1,1])$. Since $\widetilde\Phi^{-1}(0)$ does not intersect $Q_p \times [s_0,s_1] \times \{-1,1\}$ by construction, the edges of $\widetilde\Phi^{-1}(0)$ are attached to every its vertex of $0$ modulo $p$ times, except for the vertices with $s=s_0$ or $s=s_1$, corresponding to the start and the end of the homotopy.

Projecting the $1$-cycle $\widetilde\Phi^{-1}(0)$ to the rectangle $[s_0,s_1]\times [-1,1]$  and noting that every $G_p$-orbit goes to a single point under this projection, we get a $1$-dimensional cycle $S'$ modulo $p$ relative to $\{s_0,s_1\}\times [-1,1]$, intersecting a generic line $s = \mathrm{const}$ nontrivially modulo $p$ by Claim \ref{claim-nonzero-cycle}, since this is the solution set of a generic problem with a fixed segment $I(s)$. More generally, any curve connecting the bottom $[s_0,s_1]\times \{-1\}$ and the top $[s_0,s_1]\times \{1\}$ of the rectangle is homologous to such a line, and it must intersect this cycle by the homological invariance of the intersection number modulo $p$. The homological invariance holds because the rectangle $[s_0,s_1] \times[0,1]$ is a PL manifold, $S'$ has boundary on the left and the right sides of the rectangle, and we only consider the curves connecting the top and the bottom of the rectangle and not touching the sides.

We have proved the claim for a transverse $\widetilde\Phi$, and now consider the general case. Assume that we have a curve $\Gamma$ from  $[s_0,s_1]\times \{-1\}$ to $[s_0,s_1]\times \{1\}$ not touching the projection of the solution set for a not necessarily transverse $G_p$-equivariant $\widetilde\Phi$, satisfying the boundary conditions. From the compactness considerations, the minimum of $|\widetilde\Phi|$ on the preimage $\widetilde\Gamma$ of $\Gamma$ in $Q_p \times [s_0,s_1] \times [-1,1]$ is some $\varepsilon > 0$. Hence, if the transverse equivariant perturbation $\widetilde\Phi_\epsilon$ provided by Claim \ref{claim:approximation} is less than $\varepsilon$ close to $\widetilde\Phi$ then the solution set $\widetilde\Phi_\epsilon^{-1}(0)$ will still be disjoint from $\widetilde\Gamma$ and its projection to the rectangle will still be disjoint from $\Gamma$. But for a transverse $\widetilde\Phi_\epsilon$ the existence of such a curve $\Gamma$ is already shown to be impossible.
\end{proof}

Now perform the third step in our understanding of $S$, working in the full cylinder of parameters, $\mathcal I\times [-1,1]$. 

\begin{claim}
The projection $Z$ of $S$ to $\mathcal I\times [-1,1]$ separates the top $\mathcal I\times \{1\}$ from the bottom $\mathcal I \times \{-1\}$.
\end{claim}
\begin{proof}
Assume that a continuous curve
\[
\Gamma : [s_0,s_1] \to \mathcal I\times [-1,1]
\] 
passes from the bottom $\mathcal I\times \{-1\}$ to the top $\mathcal I\times \{1\}$ in the cylinder. Its first coordinate may be considered as a one-parametric family of segments $I(s)$, to which we apply Claim~\ref{claim:one-parameter} and conclude that $\Gamma$ must meet $Z$.
\end{proof}

We have the crucial separation property of $Z\subset \mathcal I\times [-1,1]$, considered as a graph of a multi-valued function. The separation property implies that this multi-valued function is nice; the first approximation to its corresponding
\[
\psi : \mathcal I\times [-1,1] \to \mathbb R
\]
could be obtained as a signed distance function, as in the proof of Lemma \ref{lemma:two-definitions}. Whenever $\psi(I, y) = 0$, the pair $(I,y)$ is in $Z$ and corresponds to some $(I,\xi,y)\in S$. By the definition of $S$, $\xi$ provides a labeled partition of $I$ into $p$ segments $I_1,\ldots, I_p$ satisfying 
\[
\phi(I_1, y) = \dots = \phi(I_p, y) = 0.
\]

But we also need to ensure that $\psi([a,a],y)\equiv y$ for all $y$, which may be not true for the signed distance function obtained from the proof of Lemma \ref{lemma:two-definitions}. We had no such difficulty in the proof of \cite[Lemma 4.2]{aak2018} and the need to overcome it here is essentially what makes this proof different.

Put for brevity $\mathcal I' = \{[a,a]\ |\ a\in [0,1]\} \subset\mathcal I$. Examining our construction of $S$ and $Z$ in case of a degenerate segment (which may only be partitioned into degenerate segments) and using the fact that $\phi([a,a],y)\equiv y$ for all $y$, we see that 
\begin{equation}
\label{equation:degeneracy}
Z\cap (\mathcal I'\times \mathbb R) = \mathcal I'\times \{0\}.
\end{equation}

In order to obtain the property $\psi([a,a], y)\equiv y$, we use a modification of the argument in the proof of Lemma \ref{lemma:two-definitions} to build the function $\psi: \mathcal I\times [-1,1]\to \mathbb R$ with zero set $Z$. Define 
\[
\psi(Z)=0,\quad \psi(\mathcal I', y) \equiv y,\quad \psi(I, 1) \equiv 1,\quad \psi(I,-1) \equiv -1.
\] 
After this and because of \eqref{equation:degeneracy} $\psi$ is continuously defined on the closed set $Y = Z \cup (\mathcal I'\times [-1,1])\cup (\mathcal I\times \{-1,1\})\supseteq Z$.

Then we extend $\psi$ by the Tietze extension theorem to the connected components of $\mathcal I\times (-1,1)\setminus Y$ so that on the connected components touching the top it remains non-negative. By adding to $\psi$ the distance function to the set $Y$ in such components (and still denoting the resulting function by $\psi$) we make $\psi$ strictly positive in top components of the complement of $Y$. We do the same on the components of the complement of $Y$ touching the bottom with the minus sign, thus extending $\psi$ to a negative function there. In effect, we obtain $\psi$ with zero set $Z$ satisfying $\psi(\mathcal I',y)\equiv y$, $\psi(\mathcal I, 1) > 0$, and $\psi(\mathcal I, -1) < 0$.

\bibliography{../Bib/karasev}

\begin{thebibliography}{10}

\bibitem{aak2018}
A.~Akopyan, S.~Avvakumov, and R.~Karasev.
\newblock Convex fair partitions into an arbitrary number of pieces.
\newblock 2018.
\newblock \href{https://arxiv.org/abs/1804.03057}{arXiv:1804.03057}, version 7
  and higher.

\bibitem{avvakumov2019}
S.~Avvakumov and R.~Karasev.
\newblock Envy-free division using mapping degree.
\newblock {\em Mathematika}, 67(1):36--53, 2021.
\newblock \href{https://arxiv.org/abs/1907.11183}{arXiv:1907.11183}.

\bibitem{bbs2009}
J.~B. Barbanel, S.~J. Brams, and W.~Stromquist.
\newblock Cutting a pie is not a piece of cake.
\newblock {\em American Mathematical Monthly}, 116(6):496--514, 2009.

\bibitem{bz2014}
P.~Blagojevi\'c and G.~Ziegler.
\newblock Convex equipartitions via equivariant obstruction theory.
\newblock {\em Israel Journal of Mathematics}, 200(1):49--77, 2014.
\newblock \href{http://arxiv.org/abs/1202.5504}{arXiv:1202.5504}.

\bibitem{bogomolnaia2020}
A.~Bogomolnaia and H.~Moulin.
\newblock Guarantees in fair division: general or monotone preferences.
\newblock 2019.
\newblock \href{http://arxiv.org/abs/1911.10009}{arXiv:1911.10009}, version 3
  and higher.

\bibitem{gale1984}
D.~Gale.
\newblock Equilibrium in a discrete exchange economy with money.
\newblock {\em International Journal of Game Theory}, 13(1):61--64, 1984.

\bibitem{hatcher2001}
A.~Hatcher.
\newblock {\em Algebraic Topology}.
\newblock Cambridge University Press, 2001.
\newblock
  \href{https://pi.math.cornell.edu/~hatcher/AT/AT.pdf}{pi.math.cornell.edu/\~hatcher/AT/AT.pdf}.

\bibitem{ahk2014}
R.~Karasev, A.~Hubard, and B.~Aronov.
\newblock Convex equipartitions: the spicy chicken theorem.
\newblock {\em Geometriae Dedicata}, 170(1):263--279, 2014.
\newblock \href{http://arxiv.org/abs/1306.2741}{arXiv:1306.2741}.

\bibitem{klartag2016}
B.~Klartag.
\newblock Convex geometry and waist inequalities.
\newblock {\em Geometric and Functional Analysis}, 27(1):130--164, 2017.
\newblock \href{http://arxiv.org/abs/1608.04121}{arXiv:1608.04121}.

\bibitem{kkm1929}
B.~Knaster, C.~Kuratowski, and S.~Mazurkiewicz.
\newblock Ein {B}eweis des {F}ixpunktsatzes f\"ur $n$-dimensionale {S}implexe.
\newblock {\em Fundamenta Mathematicae}, 14(1):132--137, 1929.

\bibitem{meunier-zerbib2018}
F.~Meunier and S.~Zerbib.
\newblock Envy-free cake division without assuming the players prefer nonempty
  pieces.
\newblock {\em Israel Journal of Mathematics}, 234:907--925, 2019.
\newblock \href{http://arxiv.org/abs/1804.00449}{arXiv:1804.00449}.

\bibitem{musin2012}
O.~R. Musin.
\newblock {B}orsuk--{U}lam type theorems for manifolds.
\newblock {\em Proceedings of the American Mathematical Society},
  140:2551--2560, 2012.

\bibitem{nandakumar2008}
R.~Nandakumar and N.~Ramana~Rao.
\newblock `{F}air' partitions of polygons -- an introduction.
\newblock 2008.
\newblock \href{http://arxiv.org/abs/0812.2241}{arXiv:0812.2241}.

\bibitem{panina2021}
G.~Panina and R.~\v{Z}ivaljevi\'c.
\newblock Envy-free division via configuration spaces.
\newblock 2021.
\newblock \href{http://arxiv.org/abs/2102.06886}{arXiv:2102.06886}.

\bibitem{segal2018}
E.~Segal-Halevi.
\newblock Fairly dividing a cake after some parts were burnt in the oven.
\newblock {\em Proceedings of the 17th International Conference on Autonomous
  Agents and MultiAgent Systems (AAMAS 2018)}, pages 1276--1284, 2018.

\bibitem{soberon2012}
P.~Sober\'on.
\newblock Balanced convex partitions of measures in $\mathbb{R}^d$.
\newblock {\em Mathematika}, 58(1):71--76, 2012.

\bibitem{steinhaus1945}
H.~Steinhaus.
\newblock Sur la division des ensembles de l'espace par les plans et des
  ensembles plans par les cercles.
\newblock {\em Fund. Math.}, 33:245--263, 1945.

\bibitem{stonetukey1942}
A.~Stone and J.~Tukey.
\newblock Generalized ``sandwich'' theorems.
\newblock {\em Duke Mathematical Journal}, 9:356--359, 1942.

\bibitem{su1999}
F.~E. Su.
\newblock Rental harmony: {S}perner's lemma in fair division.
\newblock {\em American Mathematical Monthly}, 106(10):930--942, 1999.

\end{thebibliography}

\bibliographystyle{abbrv}
\end{document}